\documentclass[11pt]{article}
\usepackage{amssymb,graphicx,amscd,color,amsmath,amsfonts,amssymb,geometry,latexsym,amsthm,units}
\usepackage[initials]{amsrefs}
\newtheorem{theorem}{Theorem}[section]

\newtheorem{corollary}[theorem]{Corollary}

\newtheorem{remark}[theorem]{Remark}

\newtheorem{lemma}[theorem]{Lemma}
\newtheorem{definition}[theorem]{Definition}

\begin{document}
\title{\textsc{Subspaces of maximal dimension contained in $L_{p}(\Omega) - \textstyle\bigcup\limits_{q<p}L_{q}(\Omega)$}}
\date{}
\author{G. Botelho\thanks{Supported by CNPq Grant 306981/2008-4 and Fapemig Grant PPM-00295-11.} \and D. Cariello \and V. V. F\'avaro\thanks{Supported by FAPEMIG Grant CEX-APQ-00208-09}\and D. Pellegrino\thanks{Supported by CNPq Grant 301237/2009-3} \and J. B. Seoane-Sep\'ulveda\thanks{Supported by by the Spanish Ministry of Science and Innovation, grant MTM2009-07848.\hfill\newline2010
Mathematics Subject Classification: 46E30, 46A45, 46A16, 46B45.\hfill\newline
}}
\maketitle
\begin{abstract}
Let $(\Omega,\Sigma,\mu)$ be a measure space and $1< p < +\infty$. In this paper we show that, under quite general conditions, the set
$L_{p}(\Omega) - \bigcup\limits_{1 \leq q < p}L_{q}(\Omega)$ is maximal spaceable, that is, it contains (except for the null vector) a
closed subspace $F$ of $L_{p}(\Omega)$ such that $\dim(F) = \dim\left(L_{p}(\Omega)\right)$. We also show that if those conditions are not fulfilled, then even the larger set $L_p(\Omega) - L_q(\Omega)$, $1 \leq q < p$, may fail to be maximal spaceable. The aim of the results presented here is, among others, to generalize all the previous work (since the 1960's) related to the linear structure of the sets $L_{p}(\Omega) - L_{q}(\Omega)$ with $q < p$ and $L_{p}(\Omega) - \bigcup\limits_{1 \leq q < p}L_{q}(\Omega)$. 
\end{abstract}
\section{Introduction and Preliminaries}
This paper is devoted to the search for what are often large linear spaces of functions enjoying certain {\em special} properties. Let $E$ be a topological vector space and let us consider such an special property $\mathcal P$. We say that the subset $M$ of $E$ formed by the vectors in $E$ which satisfy $\mathcal P$ is {\em spaceable} if $M \cup \{0\}$ contains a {\em closed} infinite dimensional subspace. The set $M$ shall be called {\em lineable} if $M \cup \{0\}$ contains an infinite dimensional linear (not necessarily closed) space.

These notions of lineability and spaceability were originally coined by V. Gurariy and they first appeared in \cite{AGS_2005,S_2006}. After the first appearance of this notion, a trend has started in which many authors became interested in this topic. Some examples of this fact are, for instance, the recent works by R. Aron (see, e.g. \cite{AGM_2001,AGS_2005,AGPS_2009,APS_2006,AS_2007}), P. Enflo (see \cite{EGS_2012}), V. Gurariy (\cite{AGS_2005,EGS_2012,GQ_2004}) or G. Godefroy (\cite{BG_2006}), just to cite some. It is important to recall that, prior to the publication of \cite{AGS_2005,S_2006}, some authors (when working with infinite dimensional spaces) already found large linear structures enjoying these type of so called ``special'' properties (even though they did not explicitly used terms like lineability or spaceability). Probably the very first result illustrating this was due to B. Levine and D. Milman (1940, \cite{LM_1940}):
\begin{theorem}
The subset of $\mathcal{C}[0,1]$ of all functions of bounded variation is not spaceable.
\end{theorem}
Later, the following analogue of this previous result was proved by V. Gurariy (1966, \cite{G_1966}):
\begin{theorem}
The set of everywhere differentiable functions on $[0,1]$ is not spaceable.
\end{theorem}
On the other hand (see also \cite{G_1966}),
\begin{theorem}
There exist closed infinite-dimensional subspaces of $\mathcal{C}[0,1]$ all of whose members are differentiable on
$(0,1)$.
\end{theorem}

Within the context of subsets of continuous functions, in 1966 V. Gurariy \cite{G_1991} showed that the set of continuous nowhere differentiable functions on $[0,1]$ is lineable. Soon after, V. Fonf,  V. Gurariy and M. Kade\v{c} \cite{FGK_1999} showed that the set of continuous nowhere differentiable functions on $[0,1]$ is spaceable in $\mathcal{C}[0,1]$. Actually, much more is known about this set. L. Rodr\'{\i}guez-Piazza \cite{R_1995} showed that the space constructed in \cite{FGK_1999} can be chosen to be isometrically isomorphic to any separable Banach space. More recently, S. Hencl \cite{H_2000} showed that any separable Banach space is isometrically isomorphic to a subspace of $\mathcal{C}[0,1]$ whose non-zero elements are nowhere approximately differentiable and nowhere H\"{o}lder. Another set that has also attracted the attention of several authors is the set of differentiable nowhere monotone functions on $\mathbb{R}$, which was proved to be lineable (see, e.g., \cite{AGS_2005,GMSS_2010}). We refer the interested reader to \cite{AGM_2001,AGZ_2000,APS_2006,AS_2007,BBFP_2012,BDP_2009,BMP_2010,BFPS_2012,G_2011,GMPS_2012,GMS_2010,GMS_2011,GGMS_2010,JMS_2012,PT_2009} for recent results and advances in this topic of lineability and spaceability, where many more examples can be found and techniques are developed in several different frameworks.

Here we shall focus on another class of function spaces, namely $L_p$ spaces, and more particularly on the sets of the form $L_p(\Omega)- \bigcup\limits_{1 \leq q < p}L_q(\Omega)$. The study of structural properties of subspaces of $L_p$ spaces is a classical topic in Banach space theory, dating back to the early days of the theory (see, e.g., \cite{livroBanach, Banach-Mazur}) up to the present days (see, e.g., \cite{BFPS_2012, BO_2012, HOS}). First of all, let us provide a clear summary and chronological overview of the series of spaceability results in this direction throughout the years.
\begin{enumerate}

\item H. Rosenthal (1968, \cite{R_1968}) showed that $c_0$ is {\it quasi-complemented} in $\ell_\infty$ (a closed subspace $Y$ of a Banach space $X$ is {\it quasi-complemented} if there is a closed subspace $Z$ of $X$ such
that $Y\cap Z=\{0\}$ and $Y+Z$ is dense in $X$); this clearly implies that $\ell_\infty - c_0$ is spaceable.

\item Later, Garc\'{\i}a-Pacheco, Mart\'{\i}n and Seoane-Sep\'ulveda proved (2009, \cite{GMS_2009}) that $\ell_\infty(\Gamma) - c_0(\Gamma)$ is spaceable for every infinite set $\Gamma$. Although it is interesting to recall that J. Lindenstrauss (1968, \cite{L_1968}) proved that, if $\Gamma$ is uncountable, then $c_0(\Gamma)$ is not quasi-complemented in $\ell_\infty(\Gamma)$.

\item In (2008, \cite{MPPS_2008}), Mu\~noz-Fern\'andez, Palmberg, Puglisi and Seoane-Sep\'ulveda proved that if $I$ is a bounded interval and $q > p \ge 1$, then $L_p(I)- L_q(I)$ is $\mathfrak{c}$-lineable. In this same paper it is proved that both, $\ell_p - \ell_q$ and $L_p(J)- L_q(J)$, are $\mathfrak{c}$-lineable for any unbounded interval $J$ and for $p>q\ge 1$.

\item One year later (2009, \cite{AGPS_2009}), Aron, Garc\'{\i}a-Pacheco, P\'erez-Garc\'{\i}a and Seoane-Sep\'ulveda showed that the linear subspaces constructed in \cite{MPPS_2008} can be chosen to be dense.

\item Bernal-Gonz\'alez (2010, \cite{Bernalstudia2010}) provided a series of conditions from which one can obtain (maximal) lineability (and dense-lineability) of the set of functions in $L_p(X, \mu)$ that are not in $L_q(X, \mu)$, where $1 \le q \neq p < \infty$ and $\mu$ denotes a regular Borel measure on a topological space $X$.
\item In (\cite{GPS_2010}, 2010) Garc\'{\i}a-Pacheco, P\'erez-Eslava and Seoane-Sep\'ulveda  proved that if $\left(\Omega, \Sigma, \mu\right)$ is a measure space such that there exists $\varepsilon > 0$ and an infinite family $\left(A_n\right)_{n \in \mathbb{N}} \subset \Sigma$ of pairwise disjoint measurable sets with $\mu\left(A_n\right) \geq \varepsilon$ for all $n \in \mathbb{N}$, then $\bigcap\limits_{p=1}^{\infty}\left( L_{\infty}(\Omega,\Sigma,\mu) - L_p (\Omega,\Sigma,\mu)\right)$ is spaceable in $L_{\infty}(\Omega,\Sigma,\mu)$ (see  \cite[Theorem 2.6]{GPS_2010}).

\item The results above, somehow, kept {\em evolving} and, in (\cite{BDFP_2011}, 2011), Botelho, Diniz, F\'avaro and Pellegrino proved (for any  Banach space $X$) that for large classes of Banach (and even quasi-Banach) spaces $E$ of $X$-valued sequences, the sets $E - \bigcup\limits_{q\in \Gamma} \ell_q(X)$ (where $\Gamma \subset [0,\infty)$), and $E - c_0(X)$ are both spaceable in $E$.

\item Next, and as a consequence of a lecture delivered by V. F\'avaro at an international conference held in Valencia (Spain) in 2010, R. Aron asked whether the result above (\cite[Corollary 1.7]{BDFP_2011}) would hold for $L_p$-spaces. This question was answered in the positive (and independently) in \cite{BFPS_2012,BO_2012}. More precisely, in \cite{BO_2012} Bernal-Gonz\'alez and Ord\'o\~nez Cabrera provided a series of conditions on a measure space $(X, \mathcal{M},\mu)$ to ensure the spaceability of the sets $L_p(\mu,X)- \bigcup\limits_{q\in [1,p)} L_q(\mu,X)$, $L_p(\mu,X)- \bigcup\limits_{q\in [p,\infty)} L_q(\mu,X)$, and $L_p(\mu,X)- \bigcup\limits_{q\in [1,\infty)- \{p\}} L_q(\mu,X)$ (for $p\ge 1$); whereas in \cite{BFPS_2012} Botelho, F\'avaro, Pellegrino and Seoane-Sep\'ulveda obtained a quasi-Banach version of this result by proving that $L_{p}[0,1] -\bigcup\limits_{q>p} L_{q}[0,1]$ is spaceable for every $p>0$.

\item In this direction it is also crucial to mention a recent paper \cite{KT_2011}, where Kitson and Timoney provided a general result from which some of the above ones (for the normed case) can be inferred.
\end{enumerate}

At this point, and after all the invested effort for the past years in looking for the ``optimal'' results on the spaceability of the sets of the form $L_p(\Omega)- L_q(\Omega)$ with $p>q$ and $L_p(\Omega)- \bigcup\limits_{1 \leq q < p}L_q(\Omega)$, we now continue with this ongoing work and provide our contributions in the form of what it is called {\em maximal}-spaceability. In other words, given a measure space $(\Omega,\Sigma,\mu)$,

\begin{quote}
{\em When does $L_{p}(\Omega) - \bigcup\limits_{1 \leq q < p} L_{q}(\Omega)$ contain, except for the null vector, a
closed subspace $F$ of $L_{p}(\Omega)$ such that $\dim(F) = \dim\left(L_{p}(\Omega)\right)$?}
\end{quote}

Of course, for the above problem to be well-posed we should have $\mu(\Omega) = +\infty$. In order to decide whether a subspace of $L_{p}(\Omega)$ has maximal dimension or not, it is of course crucial to know the dimension of $L_{p}(\Omega)$. This paper is arranged in three main sections. In Section 2 we shall study the dimension of $L_{p}(\Omega)$ and in Section 3 we shall benefit from the results proved in Section 2 to provide quite general sufficient conditions for $L_{p}(\Omega) - \bigcup\limits_{1 \leq q < p} L_{q}(\Omega)$ to be maximal spaceable. Although the results of Section 3 cover most cases, including all common $L_p(\Omega)$ spaces and some cases never studied before, in Section 4 we shall use the results of Section 2 and Section 3 to prove that even the larger set $L_p(\Omega) - L_q(\Omega)$ with $q < p$ may fail to be maximal spaceable provided that the conditions given in Section 3 are not fulfilled. By doing this we provide an ultimate answer to the spaceability of the sets of the form $L_p- \bigcup\limits_{1 \leq q < p} L_q$ for all measure spaces we are aware of. 

Many recent results concern spaceability/maximal spaceability of complements of subspaces of topological vector spaces (sometimes complements of dense subspaces). For example, \cite{KT_2011} provides quite strong results in this line. So it is important to mention that our results on the maximal spaceability of $L_{p}(\Omega) - \bigcup\limits_{1 \leq q < p} L_{q}(\Omega)$ do not require $L_q(\Omega)$ to be a subspace of $L_p(\Omega)$ for $q < p$.

Throughout this paper, $\mathbb{K}$ shall stand for either $\mathbb{R}$ or $\mathbb{C}$, $\#A$ denotes the cardinality of the set $A$, $\aleph_{0} = \#\mathbb{N}$ and $\mathfrak{c} = \#\mathbb{R}$, the continuum. The rest of the notation shall be rather usual.

\section{Computing the dimension of $L_{p}(\Omega)$}

The aim of this section is to express the dimension of $L_{p}(\Omega)$ in terms that shall be useful in the investigation of the maximal spaceability of $L_{p}(\Omega)- \bigcup\limits_{1 \leq q < p} L_{q}(\Omega)$.

In this section $(\Omega,\Sigma,\mu)$ shall denote a measure space and $0 < p < +\infty$.

\begin{definition}\rm
\begin{enumerate}
\item[(i)] $\Sigma_{fin}:=\{A\in\Sigma: \mu(A)< +\infty\}$.
\item[(ii)] Two sets $A,B\in\Sigma_{fin}$ are equivalent, denoted $A\sim B$, if
$$\mu((A-B)\cup(B-A))=0.$$
The elements of $\nicefrac{\Sigma_{fin}}{\sim}$ are denoted by $[B],$ for $B \in \Sigma_{fin}$.
\item[(iii)] The cardinal number $\#\nicefrac{\Sigma_{fin}}{\sim}$ is called the {\it entropy} of the measure space $(\Omega,\Sigma,\mu)$ and is denoted by $ent(\Omega)$.
\item[(iv)] Given a cardinal number $\zeta$, we say that the measure space $(\Omega,\Sigma,\mu)$ is $\zeta$-{\it bounded} if, for every $A\in\Sigma_{fin}$ with positive measure, there are at most $\zeta$ subsets of $A$ with positive measure belonging to different classes of $\nicefrac{\Sigma_{fin}}{\sim}$.
\item[(v)] A set $A\in\Sigma$ is an \textit{atom} if $0<\mu(A)$ and there is no $B\in\Sigma$ such that $B\subset A$ and $0<\mu(B)<\mu(A)$.
\end{enumerate}
\end{definition}

\begin{lemma} \label{lemma1}
If $ent(\Omega)\geq\aleph_{0}$, then there are sets $(B_i)_{i \in \mathbb{N}}$ in $\Sigma_{fin}$ such that $\mu(B_{i})>0$ for every $i \in\mathbb{N}$ and $\mu\left(  B_{i} \cap B_{j}\right)  = 0$ whenever $i \neq j$.
\end{lemma}

\begin{proof}
Assume first that there is a set $A_{1}\in\Sigma_{fin}$ with $\mu(A_{1})>0$
and containing no atoms. Therefore $A_{1}$ is not an atom and hence there is a
set $A_{2}\subset A_{1}$ such that $0<\mu(A_{2})<\mu(A_{1})$. By the
assumption on the existence of such $A_{1}$, we have that $A_{2}$ is not an atom either.
Repeating this argument we obtain $A_{1}\supset A_{2}\supset A_{3}\ldots$ with
$0<\mu(A_{i+1})<\mu(A_{i})$ for every $i$. Defining $B_{i}=A_{i}-A_{i+1}$ we obtain $\mu(B_{i})>0$ for every $i\in\mathbb{N}$ and $\mu\left( B_{i} \cap B_{j}\right)  = 0$ for $i \neq j$.

To complete the proof, suppose now that every $B\in\Sigma_{fin}$ with $\mu(B)>0$, contains an atom.
Let $B_1\in\Sigma_{fin}$ be an atom. Suppose that we have defined pairwise disjoint atoms $B_1, \ldots, B_k \in\Sigma_{fin}$ and let us prove that there is a measurable set $B \in\Sigma_{fin}$ such that $$\textstyle\mu\left(B-\bigcup\limits_{i=1}^k (B \cap B_i) \right)>0.$$
If we suppose that $\mu(B)=\mu\left(\bigcup\limits_{i=1}^k(B \cap B_i)\right)$ for every measurable set $B$, then
$[B]=\left[\bigcup\limits_{i=1}^k(B \cap B_i)\right]$ for every measurable set $B$. In other words, every class in $\nicefrac{\Sigma_{fin}}{\sim}$
 contains a subset of $\bigcup\limits_{i=1}^kB_i$ as a representative.
Since $B_1,\ldots,B_k$ are atoms, the only subsets of $\bigcup\limits_{i=1}^kB_i$ that belong to different equivalence classes are equivalent to either $B_1,\ldots,B_k$ or unions of some of them. In this case we have $ent(\Omega)=2^k$, which is absurd. Hence there is a measurable $B$ such that $\mu(B)=\mu\left(\bigcup\limits_{i=1}^k(B \cap B_i)\right)$, that is, $\mu\left(B-\bigcup\limits_{i=1}^k(B \cap B_i)\right)>0$. So there is an atom $B_{k+1}\subset B-\bigcup\limits_{i=1}^k(B \cap B_i)$. Therefore the sets $B_1, \ldots, B_k, B_{k+1} \in\Sigma_{fin}$ are pairwise disjoint, and, in particular, $\mu\left(  B_{i} \cap B_{j}\right)  = 0$ for $i \neq j$.
\end{proof}

Since the continuum hypothesis is not required in what follows, we would rather prefer not to assume it.

\begin{theorem}\label{thmdim}~\\
{\rm (a)} If $ent(\Omega)>\mathfrak{c}$, then $\dim \left(L_{p}(\Omega)\right)=ent(\Omega)$.\\
{\rm (b)} If $\aleph_0 \leq ent(\Omega) \leq \mathfrak{c}$, then $\dim \left(L_{p}(\Omega
)\right)=\mathfrak{c}$.\\
{\rm (c)} If $ent(\Omega)
\in\mathbb{N}$, then there is $k \in\mathbb{N}$ such that $ent(\Omega)=2^{k}$ and $\dim \left(L_{p}(\Omega)\right)=k$.
\end{theorem}

\begin{proof}
By $\chi_{A}$ we denote the characteristic function of the set $A \in \Sigma$. Let

$$W:=\left\{\sum_{i=1}^na_i\chi_{A_i}: n\in\mathbb{N},~ a_i\in\mathbb{K}\ \mbox{and}\ A_i\ \mbox{ is a representative of a class in }\ \nicefrac{\Sigma_{fin}}{\sim} \right\}.$$

By \cite[Proposition 6.7]{Folland} we know that $L_p(\Omega)=\overline{W}$. Therefore
$$\# L_p(\Omega)=\#\overline{W} \leq \#\{\mbox{Cauchy sequences in }  W\}\leq \# W^{\mathbb{N}}.$$
Assume that $ent(\Omega)\geq\mathfrak{c}$. On the one hand, $\# W= ent(\Omega)$, hence
$$\# L_p(\Omega) \leq\#\left(\nicefrac{\Sigma_{fin}}{\sim}\right)^{\mathbb{N}}=ent(\Omega).$$ On the other hand, if $A, B\in \Sigma$ are not equivalent in $\Sigma_{fin}$, then  $\chi_{A}\neq\chi_{B}$ in $L_p(\Omega)$. So  $ent(\Omega)\leq\# L_p(\Omega)$. Therefore $\# L_p(\Omega)=ent(\Omega)$.

\medskip

\noindent(a)  Since $ent(\Omega) >\mathfrak{c}$, we have $\# L_p(\Omega)=ent(\Omega)>\mathfrak{c}$. And since the cardinality of this vector space is greater than the cardinality of the scalar field, its cardinality and dimension coincide.

\medskip

\noindent (b) Since $ent(\Omega) \leq \mathfrak{c}$, again we obtain $\# L_p(\Omega)=\#W \le \mathfrak{c}$, therefore $\dim\left( L_p(\Omega)\right)\leq\mathfrak{c}$. On the other hand, since  $ent(\Omega)\geq \aleph_0$, by Lemma \ref{lemma1} there are countably many sets $B_1,B_2,\dots $ such that $\mu\left(B_i \cap B_j\right) = 0$ whenever $i \neq j$, all of them of positive measure. Choose a sequence $(a_j)_{j=1}^\infty \in \ell_p$ with $a_j > 0$ for every $j$ and define
$$f \colon \Omega \longrightarrow \mathbb{K}~,~f(x)=\sum_{i=1}^{\infty}\frac{a_j}{\mu(B_j)^{\frac{1}{p}}}\chi_{B_j}(x).$$ Notice that $\int_\Omega |f|^p\,d\mu= \sum_{i=1}^{\infty}|a_j|^p$, thus $f\in L_p(\Omega)$. Now let $\mathcal{F}$ be a totally ordered (with respect to the inclusion) family of subsets of $\mathbb{N}$ such that $\#\mathcal{F}=\mathfrak{c}$.
For example, identify $\mathbb{N}$ with $\mathbb{Q}$ and consider the family $\mathcal{F}=\{(-\infty,r)\cap\mathbb{Q}:  r\in\mathbb{R}\}$. Given $S\in\mathcal{F}$, define
$$\chi_{S}\colon \Omega\longrightarrow\mathbb{K}~,~\chi_{S}(x)=\left\lbrace\begin{array}{cc}
1 & \mbox{if \ }x\in B_j\mbox{\ with\ } j\in S\\
0 & \mbox{otherwise}
\end{array}\right.$$
Notice that $\{f\chi_S:  S\in\mathcal{F}\}$ is a linearly independent subset of $L_p(\Omega)$. Therefore $$\dim \left(L_p(\Omega)\right)\geq \# \{f\chi_S: S\in\mathcal{F}\}=\#\mathcal{F}=\mathfrak{c}.$$ It follows from the Cantor-Bernstein-Schr\"oder Theorem that $\dim\left( L_p(\Omega)\right)=\mathfrak{c}$.

\medskip

\noindent (c) Firstly let us see that, under the assumption  $ent(\Omega)\in\mathbb{N}$, every measurable set of positive measure contains an atom. In fact, otherwise we could build a sequence $A_1\supset A_2\supset\cdots$ in $\Sigma$ with $\mu(A_1)>\mu(A_2)>\cdots$. In this case, $A_i$ and $A_j$ belong to different classes whenever $i \neq j$. This is a contradiction because there are only finitely many equivalence classes.\\
\indent Let $\mathcal{S}$ be the family of all subsets of $\Sigma_{fin}$ whose elements are pairwise disjoint atoms. Consider the partial order in $\mathcal{S}$ given by the natural inclusion,
that is, for $S_{1},S_{2}\in\mathcal{S}$,
\[
S_{1}\leq S_{2}\Longleftrightarrow S_{1}\subset S_{2}.
\]
Consider a subfamily $\mathcal{S}^{\prime}=\{S_i:i\in I\}\subset \mathcal{S}$ totally ordered by inclusion, where $I$ is an index set. Hence $S=\bigcup\limits_{i\in I} S_i \in \mathcal{S}$ and $S_i\subset S$ for every $i\in I$. Then $S$ is an upper bound for $\mathcal{S}^{\prime}$. Therefore, by Zorn's Lemma there is a maximal set $U\in\mathcal{S}$ with respect to the inclusion. Since the elements of $U$ are pairwise disjoint atoms, then they are in different equivalent classes. But $ent(\Omega)
<\infty,$ so $\#U<\infty$, say $U = \{A_i: \ i=1,\ldots, k\}$ where $k\in\mathbb{N}$. Let $B\in\Sigma_{fin}$ be given. Of course $B$ can be written as the union of the following two disjoint sets: $$\textstyle B=\left(B-\bigcup\limits_{i=1}^k A_i\right)\bigcup\left(\bigcup\limits_{i=1}^k (B\cap A_i)\right).$$ Suppose that $\mu\left(B-\bigcup\limits_{i=1}^k A_i\right)>0$. In this case there is an atom $A_{k+1}$ contained in $B-\bigcup\limits_{i=1}^k A_i$ such that $A_{k+1} \cap A_i = \emptyset$ for every $i=1,\ldots, k$. Thus $\{A_{k+1}\}\cup U>U$, which contradicts the maximality of $U$. Hence $\mu\left(B-\bigcup\limits_{i=1}^k A_i\right)=0$ and so $\left[B\right]=\left[\bigcup\limits_{i=1}^k (B\cap A_i)\right]$. For each $i\in \{1,\ldots,k\}$ such that $\left[B\cap A_i\right]\neq[\emptyset]$, we have $\mu(B\cap A_i)>0$, and, since $B\cap A_i\subset A_i$ and $A_i$ is an atom, we obtain $\mu(B\cap A_i)=\mu(A_i)$, that is, $[B\cap A_i]=[A_i]$. Denoting by $J_B$ the set of all $i\in\{1,\ldots,k\}$ such that $\left[B\cap A_i\right]\neq[\emptyset]$, it follows that
\[
\mu\left(\textstyle\bigcup\limits_{i\in J_B} A_i \right) = \sum_{i\in J_B} \mu(A_i)=\sum_{i\in J_B} \mu(B \cap A_i)= \mu\left( \textstyle\bigcup\limits_{i\in J_B} (B\cap A_i )\right)=\mu (B).
\]
So every set $B\in\Sigma_{fin}$ satisfies
\begin{equation}\label{B=A_i}
[B]=\left[\bigcup\limits_{i\in J_B} A_i\right],
\end{equation}
where $J_B=\left\{i\in\{1,\ldots,k\}:\left[B\cap A_i\right]\neq[\emptyset] \right\}.$
This proves that $ent(\Omega)=2^{k}$.\\
\indent Now, we know that $L_p(X)$ is the closure of
\[
W=\left\{\sum\limits_{i=1}^n b_i\chi_{B_i} : n\in\mathbb{N},\ b_i\in\mathbb{K},\ B_i\in\Sigma_{fin}\right\}.
\]
By (\ref{B=A_i}), each $\sum\limits_{i=1}^n b_i\chi_{B_i}\in W$ is $\mu$-almost everywhere equal to an element of
\[
\left\{\sum\limits_{i=1}^ka_i\chi_{A_i} :  a_i\in\mathbb{K}\right\}.
\]
Thus
\[
L_p(X)=\overline{W}=\overline{\left\{\sum\limits_{i=1}^n b_i\chi_{B_i} : n\in\mathbb{N},\ b_i\in\mathbb{K},\ B_i\in\Sigma_{fin}\right\}}=\overline{\left\{\sum\limits_{i=1}^ka_i\chi_{A_i} : a_i\in\mathbb{K}\right\}}.
\]
Since $\dim\left\{\sum\limits_{i=1}^ka_i\chi_{A_i} :  a_i\in\mathbb{K}\right\}=k$, then $\left\{\sum\limits_{i=1}^ka_i\chi_{A_i} :  a_i\in\mathbb{K}\right\}$ is closed in $L_p(X)$. It follows that
$$\dim \left(L_p(\Omega)\right)=\dim\overline{W}=\dim\left\{\sum\limits_{i=1}^ka_i\chi_{A_i} : a_i\in\mathbb{K}\right\}=k.$$
\end{proof}

Let us state, for further reference, a fact proved in the proof above:
\begin{corollary}\label{cordimen} If $ent(\Omega) \geq \mathfrak{c}$, then $ \#L_p(\Omega) = ent(\Omega) = \dim(L_p(\Omega))$.
\end{corollary}

\begin{remark}\rm
Let us recall that the standard proof of the fact that the dimension of every infinite-dimensional Banach space is, at least, $\mathfrak{c}$ (via Baire's Theorem) depends on the Continuum Hypothesis (CH). As a byproduct, we shall now see that Theorem \ref{thmdim}, whose proof does not depend on the CH, can be used to give a CH-free proof of this fact: Let $E$ be an infinite-dimensional Banach space and let $(x_n)_{n=1}^\infty$ be a normalized basic sequence in $E$ (Mazur's classical proof of the existence of such a sequence does not depend on the CH; see \cite[Corollary 5.3]{Diestel}). The operator
$$(a_n)_{n=1}^\infty \in \ell_1 \mapsto  \sum_{n=1}^\infty a_nx_n \in E,$$
is well-defined because the series is absolutely convergent (and its linearity is obvious). The uniqueness of the representation of a vector in $E$ as a (eventually infinite) linear combination of the vectors of the basic sequence guarantees the injectivity of this linear operator. Then $\dim(E) \geq \dim\left(\ell_1\right)$. By Theorem \ref{thmdim} we know that $\dim \left(\ell_1\right) \geq \mathfrak{c}$ and, thus, $\dim (E) \geq \mathfrak{c}$.
\end{remark}

\section{$L_{p}(\Omega)- \textstyle\bigcup\limits_{q < p}L_{q}(\Omega)$ is ``{\em usually}'' maximal spaceable}

In this section we give quite general conditions under which $L_{p}(\Omega)-\bigcup\limits_{1 \leq q < p}L_{q}(\Omega)$ is maximal spaceable. It is worth mentioning once again that, unlike several results on lineability/spaceability of complements of subspaces or unions of subspaces (see, e.g., \cite{BDP_2009, BFPS_2012, BMP_2010, KT_2011}), we are not assuming that $L_{q}(\Omega) \subset L_{p}(\Omega)$.

Of course we need $L_{p}(\Omega)- \bigcup\limits_{q < p}L_{q}(\Omega) \neq \emptyset$, thus throughout this  section $(\Omega,\Sigma,\mu)$ is an infinite measure space.

The following elementary measure theoretic lemma shall be used in the proof of Lemma \ref{theo1}:

\begin{lemma}\label{elementar} Let $(B_n)_{n=1}^\infty$ be a sequence of measurable sets such that $\mu(B_n \cap B_m) = 0$ whenever $n \neq m$. Then
$$\mu\left(\textstyle\bigcup\limits_{n=1}^\infty B_n \right) = \sum_{n=1}^\infty \mu(B_n). $$
\end{lemma}

\begin{lemma}
\label{theo1} Let $\mathcal{X}$ be the set of all subsets $F$ of $\Sigma_{fin}$ satisfying the following conditions:
\begin{enumerate}
\item $\mu(A)>0$ for every $A \in F$.

\item If $A,B\in F$ are distinct, then $\mu(A\cap B)=0$.
\end{enumerate}
If the measure space $(\Omega, \Sigma, \mu)$ is $\zeta$-bounded for some cardinal number $\zeta$ with $\mathfrak{c} \le \zeta < ent(\Omega)$,  then there exists a set $G\in\mathcal{X}$ with $\# G=ent(\Omega)$.
\end{lemma}

\begin{proof}
Consider the partial order in $\mathcal{X}$ given by the natural inclusion,
that is, for $F_{1},F_{2}\in\mathcal{X}$,
\[
F_{1}\leq F_{2}\Longleftrightarrow F_{1}\subset F_{2}.
\]

Given a totally ordered subset $\mathcal{Y}$ of $\mathcal{X}$, define $F$ as the union of the elements of $\mathcal{Y}$. Since
$F\in\mathcal{X}$, $F$ is an upper bound for $\mathcal{Y}$. Thus, by Zorn's Lemma there is a maximal element $G\in\mathcal{X}$.
By assumption, each element of $G$ has at most $\zeta$ subsets with positive measure belonging to different
classes of $\nicefrac{\Sigma_{fin}}{\sim}$, then the number of subsets of elements of $G$ that represent different classes in $\nicefrac{\Sigma_{fin}}{\sim}$ is at most $(\#G)\cdot \zeta$.

Now fix $A\in\Sigma_{fin}$ with $\mu(A)>0$ and define
\[
H=\{B\in G:\mu(A\cap B)>0\}.
\]
Clearly $H\neq \emptyset$, because otherwise we would have $G\cup\{A\}>G$, which contradicts
the maximality of $G$. Let us prove that $\#H$ is at most $\aleph_{0}$. Suppose, by contradiction, that
$H$ is an uncountable set and note that, for each $B\in H$, the positive real
number $\mu(A\cap B)$ belongs to one of the sets $I_{n}:=\left[  \frac{1}{n+1},\frac{1}{n}\right]  $ or $J_{n}=\left[  n,n+1\right]  $ for some
$n\in\mathbb{N}$. There are countably many sets $I_{n},J_{n}$, with $n\in\mathbb{N}$, so it follows from the Infinite Pigeonhole
Principle that there is $n_{0}\in\mathbb{N}$ such that
\[
\mu(A\cap B)>\frac{1}{n_{0}}
\]
for uncountably many sets $B\in H$. In particular, there are distinct $(C_i)_{i \in \mathbb{N}}$ in $H$ such that
\[
\mu(A\cap C_{m})>\frac{1}{n_{0}}
\]
for every $m\in\mathbb{N}$. By Condition 2 we have that $\mu(C_{i}\cap C_{j})=0
$ whenever $i\neq j$. So Lemma \ref{elementar} gives
\[
\mu(A)\geq\mu\left(  A\textstyle\bigcap\left(  \bigcup\limits_{m=1}^{\infty
}C_{m}\right)  \right) =\sum_{m=1}^{\infty}\mu(A\cap
C_{m})=\infty,
\]
a contradiction that proves that $\#H \leq \aleph_{0}$.

Now, note that $\textstyle\bigcup\limits_{B \in H}(A \cap B) \in \Sigma_{fin}$ because $\textstyle\bigcup\limits_{B \in H}(A \cap B) \subset A$. Thus $$\mu\left(\textstyle\bigcup\limits_{B \in H}(A \cap B) \right) \leq \mu(A) < \infty.$$ Let us prove that $[A]=\left[  \bigcup\limits_{B\in H}(A\cap B)\right]  $. Assuming that
$[A]\neq\left[  \bigcup\limits_{B\in H}(A\cap B)\right]  $ in $\nicefrac{\Sigma_{fin}}{\sim}$, we have
\[
\mu\left(  A-\bigcup_{B\in H}(A\cap B)\right)  > 0.
\]
Let $C \in G$ be given and assume that
$$\mu \left(\left( A - \textstyle\bigcup\limits_{B \in H}(A \cap B)\right) \textstyle\bigcap C \right) > 0. $$
In this case,
$$\mu(A \cap C) \geq \mu \left(\left( A - \textstyle\bigcup\limits_{B \in H}(A \cap B)\right) \textstyle\bigcap C \right) > 0, $$
what implies that $C \in H$. Hence,
$$A - \textstyle\bigcup\limits_{B \in H} (A \cap B) \subset A - (A \cap C).
$$
It is clear that $(A - (A \cap C)) \cap C = (A - C) \cap C = \emptyset$, so by the inclusion above we obtain
$$\mu \left(\left( A - \textstyle\bigcup\limits_{B \in H}(A \cap B)\right) \textstyle\bigcap C \right) \leq \mu \left(( A - (A \cap C)) \textstyle\bigcap C \right) =0.$$
This contradiction proves that the intersection of $A-\bigcup\limits_{B\in H}(A\cap B)$ with each element of $G$
has null measure. Therefore $$G\cup\left\{A-\textstyle\bigcup\limits_{B\in H}(A\cap B)\right\}\in\mathcal{X}~{\rm and~}G\cup\left\{A-\bigcup\limits_{B\in H}(A\cap B)\right\}>G,$$
which contradicts the maximality of $G$. Therefore $[A]=\left[  \bigcup\limits_{B\in H}(A\cap
B)\right]  $. Since $A$ is an arbitrary set in $\Sigma_{fin}$ with positive measure, we have just proved that each class in $\nicefrac{\Sigma_{fin}}{\sim}$ can be represented by a union of countably many subsets of elements of $G$. Combining this fact with the fact that the number of subsets of elements of $G$ that represent different classes in $\nicefrac{\Sigma_{fin}}{\sim}$ is at most
$(\#G)\cdot \zeta$, we conclude that $ent(\Omega)\leq(\#G)\cdot \zeta$. By assumption we have $ent(\Omega)> \zeta \ge \mathfrak{c}$, so $ent(\Omega)\leq\#G$.

 On the other hand, we know that distinct elements of $G$ determine different classes in $\nicefrac{\Sigma_{fin}}{\sim}$. Thus $\#  G \leq ent(\Omega).$ Hence $ent(\Omega)=\#G$.
\end{proof}

\begin{lemma}\label{lemanovo}
Let $(B_i)_{i \in \mathbb{N}}$ be a sequence of pairwise disjoint measurable sets, in a measure space $(\Omega, \Sigma, \mu)$, with $0 < \mu(B_i) < \infty$ for every $i \in \mathbb{N}$. Then:\\
{\rm (a)} $\Sigma' := \left\{\bigcup\limits_{j\in J} B_j : J\subset\mathbb{N}\right\}$ is a $\sigma$-algebra of subsets of $\Omega':=\textstyle\bigcup\limits_{i=1}^{\infty} B_i$.\\
{\rm (b)} The restriction of $\mu$ to $\Sigma'$ is a measure. \\
{\rm (c)} For every $r \geq 1$,
\begin{equation} \label{destacada}L_r(\Omega')=\left\{\sum_{i=1}^{\infty}a_i\chi_{B_i}: \sum_{i=1}^{\infty}|a_i|^r\mu(B_i)<\infty\right\}.
\end{equation}
\end{lemma}

\begin{proof} (a) and (b) are straightforward. Let us prove (c). It is easy to see that any simple function having support of finite measure can be written as $\sum_{j=1}^\infty a_j \chi_{B_j}$, where only finitely many $a_j$'s are nonzero. So given $f \in L_r(\Omega')$ there is a sequence $(f_n)_{n=1}^\infty$ with $f_n = \sum_{j=1}^\infty a_n^j \chi_{B_j}$ such that only finitely many $a_n^j$'s are nonzero for every $n \in \mathbb{N}$ and $f = \displaystyle \lim_{n \rightarrow \infty} f_n$ in $L_r(\Omega')$. Fix $j \in \mathbb{N}$ for a moment. Since $\mu(B_j) < \infty$, we have $f\chi_{B_j} = \displaystyle \lim_{n \rightarrow \infty} f_n \chi_{B_j}$ in $L_r(\Omega')$. On the other hand, $f_n \chi_{B_j}(x) = a_n^j \chi_{B_j}(x)$ for every $x \in \Omega'$ and  every $n$; so $f\chi_{B_j} = \displaystyle \lim_{n \rightarrow \infty} a_n^j \chi_{B_j}$ in $L_r(\Omega')$. Hence $(a_n^j \chi_{B_j})_{n=1}^\infty$ is a Cauchy sequence in $L_r(\Omega')$, and since $0 < \mu(B_j) < \infty$ we have that $(a_n^j)_{n=1}^\infty$ is a Cauchy scalar sequence, say $a_j = \displaystyle \lim_{n \rightarrow \infty} a_n^j$. It follows easily that $a_j \chi_{B_j} = \displaystyle \lim_{n \rightarrow \infty} a_n^j \chi_{B_j}$ in $L_r(\Omega')$. The uniqueness of the limit in $L_r(\Omega')$ yields that $a_j \chi_{B_j} = f  \chi_{B_j}$ in $L_r(\Omega')$. Observing that $B_j$ contains no nonvoid measurable subset it follows that $f  \chi_{B_j}(x) = a_i  \chi_{B_j}(x)$ for every $x \in \Omega'$. In particular, $f(x) = a_j$ for every $x \in B_j$.
This holds for every $j \in \mathbb{N}$, so $f(x) = \sum_{j=1}^\infty a_j \chi_{B_j}(x)$ for every $  x\in \Omega'$. Since $\left|\sum_{j=1}^k a_j \chi_{B_j}(x) \right|\leq |f(x)|$ for every $x \in \Omega'$ and every $k \in \mathbb{N}$, by a standard application of the Dominated Convergence Theorem (see, e.g., \cite[Theorem 7.2]{Bartle}), we conclude that $f = \sum_{j=1}^\infty a_j \chi_{B_j}$ in $L_r(\Omega')$. Now it is immediate that $\sum_{j=1}^\infty |a_j|^r \mu(B_j) = \|f\|_r^r < \infty$.
\end{proof}

We shall need the following result due to Subramanian \cite{sub_1978} and Romero \cite{romero_1983} (see also \cite[Theorem 3.1]{BO_2012}):
\begin{theorem}\label{newtheo} Let $(\Omega, \Sigma, \mu)$ be a measure space and $p > q \geq 1$. Then\\
{\rm (a)} $L_p(\Omega)\supset L_q(\Omega)$ if and only if $\inf\{\mu(A) : A\in\Sigma_{fin},\mu(A)>0\}>0$.\\
{\rm (b)} $L_q(\Omega)\supset L_p(\Omega)$ if and only if $\sup\{\mu(A): A\in\Sigma_{fin}\}<\infty$.
\end{theorem}

\noindent As to the maximal spaceability of $L_{p}(\Omega) - \textstyle\bigcup\limits_{q<p}L_{q}(\Omega)$, there is nothing to do if $ent(\Omega) \in\mathbb{N}$, because in this case we have, by Theorem \ref{thmdim}(c), that $L_{p}(\Omega)$ is finite-dimensional. So we restrict ourselves, without loss of generality, to the case $ent(\Omega) \geq\aleph_{0}$.

\begin{theorem} Let $p > 1$.\label{theo2} The set $L_{p}(\Omega)-\bigcup\limits_{1 \leq q<p}L_{q}(\Omega)$ is maximal spaceable if:
\begin{itemize}
\item[\rm (a)] Either  $L_{p}(\Omega)-L_{r}(\Omega)\neq \emptyset$ for some $1\leq r<p$
and $\aleph_0 \leq ent(\Omega) \leq\mathfrak{c}$;
\item[\rm (b)] or the measure space $(\Omega, \Sigma, \mu)$ is $\zeta$-bounded for some cardinal number $\zeta$ with $\mathfrak{c} \le \zeta < ent(\Omega)$.
\end{itemize}
\end{theorem}
\begin{proof} (a) Since $\aleph_0 \leq ent(\Omega) \leq\mathfrak{c}$, by Theorem \ref{thmdim}(b) we know that $\dim\left(L_p(\Omega)\right) = \mathfrak{c}$.
Therefore we only need to prove that $L_{p}(\Omega)- \bigcup\limits_{1 \leq q<p}L_{q}(\Omega)$ is spaceable.\\

\noindent Since $L_{p}(\Omega)-L_{r}(\Omega) \neq \emptyset$ for some $1\leq r<p$, by Theorem \ref{newtheo}(b) we have that
\begin{equation} \label{eqsup}\sup\{\mu(A): A\in\Sigma_{fin}\}=\infty.\end{equation} In this case we can choose pairwise disjoint measurable sets $(B_i)_{i \in \mathbb{N}}$ such that $0<\mu(B_1)<\mu(B_{i})$ for every $i \in \mathbb{N}$. Indeed, choose $B_1 \in \Sigma_{fin}$ with $\mu(B_1) >0$ and proceed inductively in the following way: if $B_1, \ldots, B_k$ have been chosen in those conditions, by (\ref{eqsup}) there is $A_{k+1} \in \Sigma_{fin}$ such that $\mu(A_{k+1}) > 2 \mu(B_1 \cup \cdots \cup B_k)$. Choose $B_{k+1} = A_{k+1} - (B_1 \cup \cdots \cup B_k)$.\\
  \indent Consider now the measure space $(\Omega',\Sigma', \mu)$, where $\Omega'$ and $\Sigma'$ are defined as in Lemma \ref{lemanovo}.
Let us prove that $L_p(\Omega')-\bigcup\limits_{1 \leq q<p}L_q(\Omega')$ is spaceable in $L_p(\Omega')$. First note that $$\inf\left\{\mu(A): A\in\Sigma'_{fin} {\rm ~ and~}\mu(A)>0\right\}=\mu(B_1)>0.$$
From Theorem \ref{newtheo}(a) it follows that $L_p(\Omega')\supset L_q(\Omega')$ for every  $1 \leq q <p$. Applying (\ref{destacada}) for $r = q$ we know that every function in $L_q(\Omega')$ can be written as $\sum_{i=1}^{\infty}a_i\chi_{B_i}$ with $\sum_{i=1}^{\infty}|a_i|^q\mu(B_i)<\infty$.
Note that if $\left\|\sum_{i=1}^{\infty}a_i\chi_{B_i}\right\|_q<\mu(B_1)^{\frac{1}{q}}$, then  $|a_i|<1$ for every $i\in \mathbb{N}$. Since $p>q\geq 1$ and $|a_i|<1$ for every $i\in \mathbb{N}$, we have that $$\left\|\sum_{i=1}^{\infty}a_i\chi_{B_i}\right\|_q>\left\|\sum_{i=1}^{\infty}a_i\chi_{B_i}\right\|_p.$$
Given $\varepsilon > 0$, choose $\delta=\min\{\varepsilon,\mu(B_1)^{\frac{1}{q}}\}>0$. If $\left\|\sum_{i=1}^{\infty}a_i\chi_{B_i}\right\|_q<\delta$, then $$\varepsilon>\left\|\sum_{i=1}^{\infty}a_i\chi_{B_i}\right\|_q>\left\|\sum_{i=1}^{\infty}a_i\chi_{B_i}\right\|_p.$$
This shows that the inclusion $L_q(\Omega')\hookrightarrow L_p(\Omega')$ is continuous for every  $1 \leq q <p$. Choosing a sequence $(a_j)_{j=1}^\infty \in \ell_p - \bigcup\limits_{q<p} \ell_q$, it is clear that the  function $$f = \sum\limits_{j=1}^\infty \frac{a_j}{\mu(B_j)^{1/p}}\chi_{B_j} $$
 belongs to $L_p(\Omega')$. Using that $\mu(B_j) > \mu(B_1)$ for every $j$, it follows that $f \notin L_q(\Omega')$ for every $1 \leq q < p$. So $L_p(\Omega')\neq \bigcup\limits_{1 \leq q<p}L_q(\Omega')$. And since $$W:=\left\{\sum_{i=1}^na_i\chi_{A_i}: n \in \mathbb{N}, a_i\in\mathbb{K}{\rm ~and~} A_i\in\Sigma'_{fin} \right\}\subset \textstyle\bigcup\limits_{1 \leq q<p}L_q(\Omega') \subset L_p(\Omega')$$
and $W$ is dense in $L_p(\Omega')$, it follows that $\bigcup\limits_{1 \leq q<p}L_q(\Omega')$ is dense in $L_p(\Omega')$ as well. So $\bigcup\limits_{1 \leq q<p}L_q(\Omega')$ is not closed in $L_p(\Omega')$ because $L_p(\Omega')\neq \bigcup\limits_{1 \leq q<p}L_q(\Omega')$. Choose a sequence $(q_j)_{j=1}^\infty$ such that $1 \leq q_j < q_{j+1}$ for every $j$ and $q_j \longrightarrow p$. Theorem \ref{newtheo}(a) assures that $L_{q_j} \subset L_{q_{j+1}}$ for every $j$, hence
  $$\textstyle \bigcup\limits_{1 \leq q < p}L_q(\Omega') = \bigcup\limits_{j=1}^\infty L_{q_j}(\Omega').   $$The spaceability of $L_p(\Omega')-\bigcup\limits_{1 \leq q < p} L_q(\Omega')$ in $L_p(\Omega')$ follows now from \cite[Theorem 3.3]{KT_2011}.\\
\indent A function $f$ defined on $\Omega'$ shall be identified with a function defined on $\Omega$ by putting $f(x) = 0$ for every $x \in (\Omega - \Omega')$. Since $\|f\|_{L_p(\Omega')} = \|f\|_{L_p(\Omega)}$ for every $f \in L_p(\Omega')$, it is plain that $L_p(\Omega')$ is a closed subspace of $L_p(\Omega)$ up to this identification.\\
 \indent Use that $L_p(\Omega')\supset \textstyle\bigcup\limits_{1 \leq q<p} L_q(\Omega')$ and apply (\ref{destacada}) for $r = p$ and for $r = q <p$ to conclude that
 $$L_p(\Omega')\textstyle\bigcap \left(\textstyle\bigcup\limits_{1 \leq q<p} L_q(\Omega)\right)= \textstyle\bigcup\limits_{1 \leq q<p}L_q(\Omega').$$
Thus
\begin{align*}L_p(\Omega')- \textstyle\bigcup\limits_{1 \leq q<p}L_q(\Omega')&=L_p(\Omega')-\left(L_p(\Omega')\bigcap \left(\textstyle\bigcup\limits_{1 \leq q<p} L_q(\Omega)\right)\right)\\&=L_p(\Omega')-\textstyle\bigcup\limits_{1 \leq q<p} L_q(\Omega).
\end{align*}
 It follows that $L_p(\Omega')- \textstyle\bigcup\limits_{1 \leq q<p}L_q(\Omega)$ is spaceable in the closed subspace $L_p(\Omega')$ of $L_p(\Omega)$, hence $L_p(\Omega')- \textstyle\bigcup\limits_{1 \leq q<p} L_q(\Omega)$ is spaceable in $L_p(\Omega)$. Therefore $L_p(\Omega)- \textstyle\bigcup\limits_{1 \leq q<p} L_q(\Omega)$ is spaceable in $L_p(\Omega)$.

\medskip

\noindent (b) Let $G$ be the family whose existence is guaranteed by Lema \ref{theo1}. Since $\#G=ent(\Omega)>\zeta$ and there are only $\mathfrak{c}$  possible values for the measures of the sets in $G$ (of course
$\mu(B)\in(0,\infty)$ for every $B \in G$), there is a subfamily $G^{\prime}\subset G$, with the same
cardinality of $G$, such that all members of $G'$ have the same measure (this is another application of the Infinite Pigeonhole Principle). Denote $G^{\prime}=\{A_{k}:k\in I\}$ with $\#I=ent(\Omega)$.
Recall that $A_{k}\neq A_{s}$ implies $\mu(A_{k}\cap A_{s})=0$ but
$\mu(A_{k})=\mu(A_{s})$. Since the cardinality of $I$ is greater than $\zeta$ and $\aleph
_{0}\cdot \zeta=\zeta$, for every $i\in I$ and every $n\in\mathbb{N}$ there
is a set $A_{i}^{n}$ so that:\\
\indent(i) $A_{i}^{j}\neq A_{i}^{k}$ whenever $i \in I$ and $j\neq k$ are
positive integers;\\
\indent (ii) The sets $J_{i}:=\{A_{i}^{j}:j\in\mathbb{N}\}$, $i \in I$, are pairwise disjoint;\\
\indent(iii) $
G^{\prime}={\textstyle\bigcup\limits_{i\in I}} J_{i}$.\\

Select a sequence $(b_{j})_{j=1}^{\infty}\in \ell_{p}-\textstyle\bigcup\limits_{q<p}\ell_{q}$ with $b_j >0$ for every $j$. For each $k\in I$, define
\[
f_{k}=\sum_{j=1}^{\infty}b_{j}\chi_{A_{k}^{j}},
\]
Observe that
\begin{enumerate}
\item The intersection of the supports of $f_{k}$ and $f_{s}$, $k \neq s$, has measure zero.
Therefore $\#\{f_{k} : k\in I\}=\# I$ and the functions $f_{k}$'s are linearly independent.
\item Let $k \in I$ and $i \in \mathbb{N}$. Since $\mu(A_{k}^{j}\cap A_{k}^{s})=0$ for all all positive integers $j \neq s$ and $\mu(A_k^j) = \mu(A_k^i)$ for every $j \in \mathbb{N}$, for every $t >0$ we have
\begin{equation}\label{eqnew}\int|f_{k}|^{t}\,d\mu=\sum_{j=1}^{\infty}|b_{j}|^{t}
\mu(A_{k}^{j})=\left(  \sum_{j=1}^{\infty}|b_{j}|^{t}\right)
 \mu(A_{k}^{i}).\end{equation}   Therefore each
$f_{k}\in L_{p}(\Omega)-\textstyle\bigcup\limits_{q<p}L_{q}(\Omega)$.
\item For all $k,l \in I$, \begin{equation}\label{equnew}\int|f_{k}|^{p}\,d\mu=\int|f_{l}|^{p}\, d\mu.\end{equation}
\end{enumerate}

Let $W={\rm span}\{f_{k} : k\in I\}\subset L_{p}(\Omega)$. Let $(h_{n})_{n=1}^\infty$ be a Cauchy sequence
in $W$ (with respect to the $L_{p}(\Omega)$-norm). Each $h_{n}$ is a finite linear combination of some
$f_{k}$'s, so these functions altogether demand only countably many $f_{k}$'s in their representations as linear combinations. Let $(g_l)_{l=1}^\infty$ be an enumeration of these $f_{k}$'s. Thus
\[
h_{n}=\sum_{l=1}^{\infty} a_{l}^{n}g_{l},
\]
where, for each $n$, only finitely many $a_{l}^{n}$'s are nonzero. Using that the intersection of the supports of $g_{k}$ and $g_{s}$, $k \neq s$, has measure zero and (\ref{equnew}) we obtain, for any fixed $j \in \mathbb{N}$,
\[
\int|h_{n}-h_{s}|^{p}\, d\mu=\sum_{l=1}^{\infty}|a_{l}^{n}-a_{l}^{s}|^{p}\cdot \int|g_{l}|^{p} \, d\mu=\left(\sum_{l=1}^{\infty}|a_{l}^{n}-a_{l}^{s}|^{p} \right)\int|g_{j}|^{p}\, d\mu.
\]
It follows that $\left((a_{l}^{n})_{l=1}^\infty\right)_{n=1}^\infty$ is a Cauchy sequence in $\ell_{p}$, say $\displaystyle \lim
_{n\rightarrow\infty}\left(a_{l}^{n}\right)_{l=1}^\infty= (a_{l})_{l=1}^\infty\in \ell_{p}$. Define $h=\sum_{l=1}^{\infty}a_{l}g_{l}$ and notice that $h\in L_{p}(\Omega)$. Now
\[
\int|h_{n}-h|^{p}\, d\mu=\sum_{l=1}^{\infty}|a_{l}^{n}-a_{l}|^{p} \cdot\int|g_{l}|^{p} \, d\mu.
\]
Since $\displaystyle\int|g_{l}|^{p}\, d\mu$ does not depend on $l$, by (\ref{eqnew})  and
$\displaystyle \lim
_{n\rightarrow\infty}\left(a_{l}^{n}\right)_{l=1}^\infty= (a_{l})_{l=1}^\infty$ in $\ell_{p}$, we obtain $\displaystyle \lim_{n\rightarrow\infty}h_{n}= h$ in
$L_{p}(\Omega)$. Finally, if $h\neq0$ then some $a_{l}\neq0$, hence $h\notin \textstyle\bigcup\limits_{q<p} L_{q}(\Omega)$.
\end{proof}

\begin{remark}\rm
Observe that in case (b) of the theorem above we have actually proved that $L_{p}(\Omega)-\bigcup\limits_{0 < q<p}L_{q}(\Omega)$ is maximal spaceable for every $p > 0$. Notice that, as a particular case, from Theorem \ref{newtheo} (b) one has that condition \eqref{eqsup} is fulfilled and, thus, we also obtain (independently) a result already given in \cite{BO_2012} on the spaceability of this set.
\end{remark}

All usual infinite measure spaces satisfy either condition $(a)$ of Theorem \ref{theo2} or condition $(b)$ of Theorem \ref{theo2} with $\zeta = \mathfrak{c}$ (for instance, a concrete example of an infinite measure space satisfying condition $(b)$ is a set of cardinality greater than $\mathfrak{c}$ endowed with the counting measure).

\section{$L_{p}(\Omega)-L_{q}(\Omega)$ may fail to be maximal spaceable for $p>q$}

As we have proved in the previous section, $L_{p}(\Omega)-\bigcup\limits_{1 \leq q<p}L_{q}(\Omega)$ is maximal spaceable in most cases. Nevertheless, in this section we prove that there exist (quite exotic) infinite measure spaces $(\Omega, \Sigma, \mu)$ such that the larger set $L_{p}(\Omega)-L_{q}(\Omega)$, $ q < p$, fails to be maximal spaceable. Actually we prove much more: given $1 \leq  q < p $ and cardinal numbers $\kappa > \zeta \geq \mathfrak{c}$, we construct an infinite measure space $(\Omega, \Sigma, \mu)$ such that:\\
 (i) $\dim(L_p(\Omega)) = \kappa$;\\
  (ii) $\zeta$ is the maximal dimension of a closed subspace of $L_p(\Omega)$ contained (except for the null vector) in $L_p(\Omega) - L_q(\Omega)$.   

%
%
%

It is worth mentioning that the construction depends on the results of Sections 2 and 3.

\begin{lemma}\label{lemaum} Let $\zeta$ be a cardinal number such that $\zeta \geq \mathfrak{c}$, let $X_\zeta$ be a set such that $\#X_\zeta = \zeta$ and let the set ${\cal P}(X_\zeta)$ of all subsets of $X_\zeta$ be endowed with the counting measure. Then $\dim(L_p(X_\zeta)) = ent(X_\zeta)= \zeta$ for every $0 < p < \infty$.
\end{lemma}

\begin{proof} For $A \subset X_\zeta$, it is clear that $A\in {\cal P}(X_\zeta)_{fin}$ if and only if $\# A<\infty$, so $\# {\cal P}(X_\zeta)_{fin}= \# X_\zeta$. It follows that $ent(X_\zeta) =\# \nicefrac{{\cal P}(X_\zeta)_{fin}}{\sim}\leq \#X_\zeta$. On the other hand, different  singletons belong to different classes in $ {\cal P}(X_\zeta)_{fin}$, therefore $\# X_\zeta\leq ent(X_\zeta)$. Combining this with $ent(X_\zeta) = \# X_\zeta = \zeta \geq \mathfrak{c}$, by Theorem \ref{thmdim} we have $\dim(L_p(X_\zeta)) = ent(X_\zeta) = \# X_\zeta = \zeta$.
%
\end{proof}

The key to the proof of the following lemma was communicated to the authors by L. Bernal-Gonz\'alez.

\begin{lemma}\label{lemabernal} For every cardinal number $\kappa \geq \mathfrak{c}$ there exists a probability space $(T_\kappa, \Sigma_\kappa, \mu_\kappa)$ such that $\dim(L_p(T_\kappa)) = ent(T_\kappa) = \kappa$ for every $0 < p < \infty$.
\end{lemma}

\begin{proof} Let $\Gamma$ be a set with $\# \Gamma= \kappa$. Let $T_\kappa$ be the product of $\kappa$ copies of $[0,1]$, that is $T_\kappa = \displaystyle\prod_{\gamma \in \Gamma}[0,1]$, and let $\Sigma_\kappa$ be the product $\sigma$-algebra of the Borel $\sigma$-algebra on $[0,1]$, that is, the $\sigma$-algebra on $T_\kappa$ generated by the inverse images of Borel subsets of $[0,1]$ by the projections in each coordinate (cf. \cite[Definition 9.1]{bauer}, \cite[Definition 22.2]{Hewitt-Stromberg}). 
By \cite[Section 22]{Hewitt-Stromberg} there exists a probability measure $\mu_\kappa$ on $\Sigma_\kappa$ such that if $A = \displaystyle \prod_{\gamma \in\Gamma}A_\gamma$, where $A_\gamma = [0,1]$ except for $\gamma = \gamma_i$, $i = 1, \ldots, n$, then $\mu_\kappa(A) = m(A_{\gamma_1}) \cdot \cdots \cdot m(A_{\gamma_n})$, where $m$ is the Lebesgue measure. Since $\kappa \geq \mathfrak{c}$, $\Sigma_\kappa$ is generated by $\kappa \times \mathfrak{c} = \kappa$ sets, by \cite[Problem 23, Chapter 12]{Komjath} it follows that
$\# \Sigma_\kappa = \kappa$ and, {\it a fortiori}, $ent(T_\kappa) \leq \kappa$. On the other hand, for $\gamma_i, \gamma_j \in \Gamma$, $\gamma_i \neq \gamma_j$, setting $A_{\gamma_i} = B_{\gamma_j} = [0, \frac12 ]$, the sets $A = \displaystyle \prod_{\gamma \in\Gamma}A_\gamma$, where $A_\gamma = [0,1]$ for every $\gamma \neq \gamma_i$, and $B = \displaystyle \prod_{\gamma \in\Gamma}B_\gamma$, where $B_\gamma = [0,1]$ for every $\gamma \neq \gamma_j$, belong to different classes in $\nicefrac{(\Sigma_\kappa)_{fin}}{\sim}$. This shows that $\kappa \leq ent(T_\kappa)$. By Theorem \ref{thmdim} we have $\dim(L_p(T_\kappa)) = ent(T_\kappa) = \kappa$.
\end{proof}

\begin{definition}\label{defnova} \rm Let $\zeta, \kappa \geq \mathfrak{c}$ be cardinal numbers. Consider the measure spaces \\$(X_\zeta, {\cal P}(X_\zeta), \nu)$ of Lemma \ref{lemaum}, where $\nu$ is the counting measure, and $(T_\kappa, \Sigma_\kappa, \mu_\kappa)$ of Lemma \ref{lemabernal}. Of course $X_\zeta$ can be chosen in such a way that $X_\zeta \cap T_\kappa = \emptyset$. Consider the measure space $(Y, \mathcal{A}, \lambda)$, where

\begin{enumerate}
\item[1)] $Y =  T_\kappa \cup X_\zeta $,
\item[2)] $\mathcal{A} =\{B\cup C : B\in \Sigma_\kappa \mathrm{~and~} C\in{\cal P}(X_\zeta)\}$, and
\item[3)] $\lambda(B\cup C)=\mu_\kappa(B)+\nu(C)$ for all $B \in\Sigma_\kappa$ and $C \in{\cal P}(X_\zeta)$.
\end{enumerate}
\end{definition}

A subset $A$ of a topological vector space $E$ is {\it $\eta$-lineable} ({\it $\eta$-spaceable}, respectively), where $\eta$ is a cardinal number, if $A \cup \{0\}$ contains a (closed, respectively) $\eta$-dimensional subspace of $E$.

\begin{theorem} Let $\zeta, \kappa$ be cardinal numbers such that $\kappa > \zeta \geq \mathfrak{c}$, let $(Y, \mathcal{A}, \lambda)$ be the measure space of Definition \ref{defnova} and let $1 \leq q < p$. Then:\\
{\rm (i)} $\dim(L_p(Y)) = \kappa$;\\
 {\rm (ii)} $L_p(Y) - L_q(Y)$ is $\zeta$-spaceable but is not $\eta$-lineable for any cardinal number $\eta > \zeta$.

 In particular, $L_p(Y) - L_q(Y)$ fails to be maximal spaceable.
\end{theorem}

\begin{proof} (i) By Lemmas \ref{lemaum} and \ref{lemabernal} we have
$$ent(Y) = ent(T_\kappa) \times ent(X_\zeta) = \kappa \times \zeta = \kappa $$
because $c \leq \zeta < \kappa$. Thus $\dim(L_p(Y)) = ent(Y) = \kappa$ by Theorem \ref{thmdim}.\\
(ii) Of course each function $0 \neq f \in L_p(Y)$ can be written as $f=f\cdot \chi_{T_\kappa}+ f\cdot\chi_{X_\zeta}$. Assume, for a while, that there is a subspace $V$ of dimension greater than $\zeta$ inside $(L_p(Y)-L_q(Y))\cup\{0\}$. In that case, consider the projection $$\pi \colon V\longrightarrow L_p(X_\zeta)~,~\pi(f)=f|_{X_\zeta}.$$
So, $$V=\textstyle\bigcup\limits_{g\in \pi(V)} \pi^{-1}(\{g\}).$$
By Lemma \ref{lemaum} we know that $ent(X_\zeta)= \zeta \geq \mathfrak{c}$, thus $\# L_p(X_\zeta)= ent(X_\zeta) =\zeta$ by Corollary \ref{cordimen}. The dimension of $V$ being greater than $\zeta$ implies that the cardinality of $V$ is also greater than $\zeta$. But $V$ is the union of at most $\zeta$ sets of the form $\pi^{-1}(\{g\})$ because $$\# \pi(V) \leq \#L_p(X_\zeta)= \zeta.$$ So there is $g \in \pi(V)$ such that the set $\pi^{-1}(\{g\})$ has cardinality greater than 1. Then there are $f,h \in V$, $h\neq f$, such that $\pi(f)=g=\pi(h)$, hence $f\cdot \chi_{X_\zeta}=h\cdot \chi_{X_\zeta}.$
Finally, $$0\neq f-h=f \cdot \chi_{T_\kappa}- h \cdot\chi_{T_\kappa} = (f-h)\cdot \chi_{T_\kappa}.$$
We know that $(f - h) \in L_p(Y)$, so $(f-h)\cdot \chi_{T_\kappa} \in L_p(T_\kappa)$. Since $\mu_\kappa(T_\kappa)=1$, by Theorem \ref{newtheo}(b) we have $L_p(T_\kappa) \subset L_q(T_\kappa)$. So $(f-h)\cdot \chi_{T_\kappa} \in L_q(T_\kappa)$, therefore $(f-h) = (f-h)\cdot \chi_{T_\kappa}\in L_q(Y)$. But $V$ is a linear subspace, so $(f - h) \in V$, which is not possible because $V\subset (L_p(Y)-L_q(Y))\cup\{0\}$.
So there is no subspace $V$ of dimension greater than $\zeta$ inside $L_p(Y)-L_q(Y)$.

Now let us prove that there is a closed $\zeta$-dimensional subspace of $L_p(Y)$ inside $(L_p(Y)-L_q(Y))\cup\{0\}$. If $\mathfrak{c}= \zeta$, then $ent(X_\zeta)=\mathfrak{c}$, so $(L_p(X_\zeta)-L_q(X_\zeta))\cup\{0\}$ contains a closed $\dim(L_p(X_\zeta))$-dimensional subspace $V$ of $L_p(X_\zeta)$ by Theorem \ref{theo2}(a). And if $\mathfrak{c}< \zeta$, then $ent(X_\zeta)=\zeta>\mathfrak{c}$. Since every set of finite measure in $X$ is a finite set, we conclude that $X_\zeta$ is $\mathfrak{c}$\,-bounded. In this case, $(L_p(X_\zeta)-L_q(X_\zeta))\cup\{0\}$ contains a closed $\dim(L_p(X_\zeta))$-dimensional subspace $V$ of $L_p(X)$ by Theorem \ref{theo2}(b).

Therefore, in any case there is a closed $\dim(L_p(X_\zeta))$-dimensional subspace $V$ of $L_p(X_\zeta)$ inside $(L_p(X_\zeta)-L_q(X_\zeta))\cup\{0\}$, with $\dim(V)=M$. It is plain that the correspondence
$$f \in L_p(X_\zeta) \mapsto \widetilde f \in L_p(Y)~,~\widetilde f(x) = \left\{\begin{array}{cl} f(x), & {\rm ~if~} x\in X_\zeta, \\ 0, & {\rm ~if~}  x \in T_{\kappa},  \end{array} \right. $$ is a linear embedding, so $L_p(X_\zeta)$ can be regarded as a closed subspace of $L_p(Y)$. By Theorem \ref{newtheo}(a) we know that $L_q(X_\zeta) \subseteq L_p(X_\zeta)$, so $L_p(X_\zeta)\cap L_q(Y)=L_q(X_\zeta)$. It follows that $L_p(X_\zeta)-L_q(X_\zeta)\subset L_p(Y)-L_q(Y)$. Therefore there is a copy of $V$ inside $(L_p(Y)-L_q(Y))\cup\{0\}$.
\end{proof}

\bigskip

\noindent{\bf Acknowledgement.} The authors thank L. Bernal-Gonz\'alez for showing us how to prove Lemma \ref{lemabernal} and for many other helpful suggestions.
%
%
%
%



\bigskip

\noindent Geraldo Botelho, Daniel Cariello, and Vin\'icius V. F\'avaro\\
Faculdade de Matem\'{a}tica\\
Universidade Federal de Uberl\^{a}ndia\\
38.400-902 -- Uberl\^{a}ndia -- Brazil\\
e-mails: botelho@ufu.br, dcariello@famat.ufu.br, vvfavaro@gmail.com

\bigskip

\noindent Daniel Pellegrino\\
Departamento de Matem\'{a}tica\\
Universidade Federal da Para\'{\i}ba\\
58.051-900 -- Jo\~{a}o Pessoa -- Brazil\\
e-mail: dmpellegrino@gmail.com

\bigskip

\noindent Juan B. Seoane-Sep\'ulveda\\
Departamento de An\'{a}lisis Matem\'{a}tico\\
Facultad de Ciencias Matem\'{a}ticas\\
Plaza de Ciencias 3\\
Universidad Complutense de Madrid\\
Madrid -- 28040 -- Spain\\
e-mail: jseoane@mat.ucm.es

\end{document}